\documentclass{article}
\usepackage{amssymb,amsmath,amsthm,mathrsfs,tikz-cd, comment}
\usepackage[title]{appendix}

\title{A function field analogue of Ligozat’s theorem for Drinfeld modular units}
\author{Sheng-Yang Kevin Ho}

\date{}

\newcommand{\Gal}{\operatorname{Gal}}
\newcommand{\SL}{\operatorname{SL}}
\newcommand{\GL}{\operatorname{GL}}

\newcommand{\Div}{\operatorname{Div}}

\newcommand{\divisor}{\operatorname{div}}
\newcommand{\cusp}{\operatorname{cusp}}
\newcommand{\ord}{\operatorname{ord}}

\newcommand{\CC}{\mathcal{C}}
\newcommand{\TT}{\mathcal{T}}

\newcommand{\p}{\mathfrak{p}}
\newcommand{\q}{\mathfrak{q}}
\newcommand{\n}{\mathfrak{n}}
\newcommand{\m}{\mathfrak{m}}

\newcommand{\pdiv}{\mid\!\mid}

\theoremstyle{plain}
\newtheorem{thm}{Theorem}[section]
\newtheorem{lem}[thm]{Lemma}
\newtheorem{conj}[thm]{Conjecture}

\newtheorem{prop}[thm]{Proposition}

\theoremstyle{definition}
\newtheorem{defn}[thm]{Definition}
\newtheorem{Ex}[thm]{Example}

\theoremstyle{remark}
\newtheorem{rem}[thm]{Remark}

\NewDocumentCommand{\tens}{e{_^}}{%
  \mathbin{\mathop{\otimes}\displaylimits
    \IfValueT{#1}{_{#1}}
    \IfValueT{#2}{^{#2}}
  }%
}

\begin{document}

\maketitle

\begin{abstract}
Fix a nonzero level $\mathfrak{n} \in \mathbb{F}_q[T]$. In this paper, we first establish a function field analogue of Ligozat’s theorem, which serves as our main result and provides a criterion for Drinfeld modular units on the Drinfeld modular curve $X_0(\mathfrak{n})$. We further conjecture that this criterion characterizes all Drinfeld modular units; we verify the conjecture in the cases of prime power level and of level equal to the product of two primes.

Second, as an application of Drinfeld modular units, we investigate the rational cuspidal divisor class group $\mathcal{C}(\mathfrak{n})$ of $X_0(\mathfrak{n})$. We construct an injective map $g$ from the group of degree $0$ rational cuspidal divisors on $X_0(\mathfrak{n})$ to the group of Drinfeld modular units on $X_0(\mathfrak{n})$ tensored with $\mathbb{Q}$ over $\mathbb{Z}$. As a result, we establish an explicit upper bound for the exponent of $\mathcal{C}(\mathfrak{n})$ for general level $\n$.
\end{abstract}

\section{Introduction}
Classically, Ligozat’s theorem, established in $1975$, completely describes the modular units on the modular curves $X_0(N)$, $N\in \mathbb{N}$, in terms of the eta quotients. Here, $X_0(N)$ is the modular curve over $\mathbb{Q}$ with $X_0(N)_{\mathbb{C}} = \Gamma_0(N)\backslash (\mathbb{H}\cup \mathbb{P}^1(\mathbb{Q}))$, where $\mathbb{H}$ is the complex upper half plane and $\Gamma_0(N)$ is the congruence subgroup of $\SL_2(\mathbb{Z})$ consisting of matrices that are upper-triangular modulo $N$. Recall that the Dedekind eta function is defined on $\mathbb{H}$, as a $24$-th root of the classical discriminant function, by
$$\eta(z) := e^{\frac{\pi i z}{12}}\prod_{n=1}^{\infty}(1-e^{2\pi i n z}) = q^{\frac{1}{24}}\prod_{n=1}^{\infty}(1-q^n),$$ where $q = e^{2\pi i z}$ and $z\in \mathbb{H}$. The eta quotients of level $N$ are
$$g_{\underline{r}}(z) := \prod_{0 < a|N}\eta_a(z)^{r_a},$$ where $\eta_a(z) := \eta(az)$, and $\underline{r} = (r_a)$ is a family of integers $r_a\in\mathbb{Z}$ indexed by $0 < a\mid N$.

\begin{thm} (Ligozat {\cite[Proposition 3.2.1]{Ligozat_1975}})
Let $N$ be a positive integer. Then $g_{\underline{r}}$ is a modular function on $X_0(N)$, i.e., a meromorphic function on $\mathbb{H}\cup \mathbb{P}^1(\mathbb{Q})$ invariant under the action of $\Gamma_0(N)$, if and only if the following conditions are true:
\begin{enumerate}
\item $\displaystyle\sum_{0 < a|N} r_{a} \cdot a\equiv 0 \mod 24$.
\item $\displaystyle\sum_{0 < a|N} r_{a} \cdot (N/a)\equiv 0 \mod 24$.
\item $\displaystyle\sum_{0 < a|N} r_{a} = 0$.
\item $\displaystyle\prod_{0 < a|N} a^{r_a}\in (\mathbb{Q}^\ast)^2$ (the square of a rational number).
\end{enumerate}
\end{thm}

\begin{rem}
A modular function on $X_0(N)$ which does not have zeros or poles on $\mathbb{H}$ is called a modular unit on $X_0(N)$. Since the Dedekind eta function $\eta(z)$ is non-vanishing and holomorphic on $\mathbb{H}$, any modular function $g_{\underline{r}}$ on $X_0(N)$ obtained in the above theorem is a modular unit on $X_0(N)$.
\end{rem}

In the function field analogue, let $\mathbb{F}_q$ be a finite field of characteristic $p$ with $q$ elements. Let $A = \mathbb{F}_q[T]$ denote the polynomial ring in the indeterminate $T$ over $\mathbb{F}_q$, and $K = \mathbb{F}_q(T)$ represent the rational function field. Denote $A_+$ the set of monic (non-zero) polynomials in $A$. An irreducible element in $A_+$ is called a prime. Define $K_\infty = \mathbb{F}_q((\pi_\infty))$ as the completion of $K$ at the infinite place, where $\pi_\infty := T^{-1}$. Let $|\cdot| = |\cdot|_\infty$ denote the normalized absolute value on $K_\infty$ with $|T|_\infty := q$. Let $\mathbb{C}_\infty$ denote the completion of an algebraic closure of $K_\infty$. Define $\Omega = \mathbb{C}_\infty - K_\infty$ as the Drinfeld upper half plane. 

Fix $\n\in A_+$. The level-$\n$ Hecke congruence subgroup of $\GL_2(A)$ is
$$\Gamma_0(\n):=\left\{\begin{pmatrix}
    a & b\\ c & d
\end{pmatrix}\in \GL_2(A)~\middle|~c\equiv 0~\text{mod}~\n\right\}.$$ Let $\Gamma_0(\n)$ act on $\Omega$ by linear fractional transformations. Drinfeld proved in \cite{drinfeld_elliptic_1974} that the quotient $\Gamma_0(\n)\backslash \Omega$ is the space of $\mathbb{C}_\infty$-points of an affine curve $Y_0(\n)$ defined over $K$, which is a moduli space of rank-$2$ Drinfeld modules. The unique smooth projective curve over $K$ containing $Y_0(\n)$ as an open subvariety is denoted by $X_0(\n)$, which is called the Drinfeld modular curve of level $\n$. Denote $\mathcal{O}(\Omega)^\ast$ the group of non-vanishing holomorphic rigid-analytic functions on $\Omega$, and let $\Delta(z)\in \mathcal{O}(\Omega)^\ast$ be the Drinfeld discriminant function defined in \cite[p. 183]{gekeler_1997}. For $\mathfrak{a}\in A_+$, denote $\Delta_{\mathfrak{a}}(z):=\Delta(\mathfrak{a}z)\in \mathcal{O}(\Omega)^\ast$.

We establish the following theorem, which is our main result.
\begin{thm} [A function field analogue of Ligozat’s theorem]\label{Main Theorem}
Fix $\n\in A_+$. Let $$g_{\underline{r}}(z) := \prod_{\substack{\mathfrak{a}|\n\\ \mathfrak{a}\in A_+}}\Delta_{\mathfrak{a}}^{r_{\mathfrak{a}}},$$ called a Delta quotient on $X_0(\n)$, where $\underline{r} = (r_{\mathfrak{a}})$ is a family of integers $r_{\mathfrak{a}}\in\mathbb{Z}$ indexed by monic $\mathfrak{a}\mid \n$. Then $g_{\underline{r}}$ has a $(q-1)(q^2-1)$th root as a modular function on $X_0(\n)$, i.e., a meromorphic function on $\Omega\cup \mathbb{P}^1(K)$ invariant under the action of $\Gamma_0(\n)$, if and only if the following conditions are true:
\begin{enumerate}
\item $\displaystyle\sum_{\substack{\mathfrak{a}|\n\\ \mathfrak{a}\in A_+}} r_{\mathfrak{a}} \cdot |\mathfrak{a}|\equiv 0 \mod (q-1)(q^2-1)$.
\item $\displaystyle\sum_{\substack{\mathfrak{a}|\n\\ \mathfrak{a}\in A_+}} r_{\mathfrak{a}} \cdot |\n/\mathfrak{a}|\equiv 0 \mod (q-1)(q^2-1)$.
\item $\displaystyle\sum_{\substack{\mathfrak{a}|\n\\ \mathfrak{a}\in A_+}} r_{\mathfrak{a}} = 0$.
\item $\displaystyle\prod_{\substack{\mathfrak{a}|\n\\ \mathfrak{a}\in A_+}} {\mathfrak{a}}^{r_{\mathfrak{a}}}\in (K^\ast)^{q-1}$.
\end{enumerate}
\end{thm}

\begin{rem}
A modular function on $X_0(\n)$ which does not have zeros or poles on $\Omega$ is called a (Drinfeld) modular unit on $X_0(\n)$. Since the Drinfeld discriminant function $\Delta(z)$ is non-vanishing and holomorphic on $\Omega$, a modular function as a $(q-1)(q^2-1)$th root of $g_{\underline{r}}$ on $X_0(\n)$ obtained in the above theorem is a modular unit on $X_0(\n)$.
\end{rem}

The main difficulty of proving Theorem \ref{Main Theorem} is that the Drinfeld discriminant function $\Delta(z)$ only has a maximal $(q-1)$-th root in $\mathcal{O}(\Omega)^\ast$ by \cite[Corollary 2.10]{gekeler_1997}. If one tries to find a $(q-1)(q^2-1)$-th root (up to constant multiple) of $\Delta(z)$, there is only a formal product in $t^{\frac{1}{q^2-1}}$ by \cite{GEKELER_product_expansion} and \cite[(6.2)]{Gekeler1984}:
$$\widetilde{\eta}(z) := t^{\frac{1}{q^2-1}} \prod_{\substack{\mathfrak{a}\in A_+}}f_\mathfrak{a}(t),$$ which is not in $\mathcal{O}(\Omega)^\ast$. In \cite{GEKELER_product_expansion}, $\widetilde{\pi}A$ is the Carlitz period, $t := t(z) := \exp_{\widetilde{\pi}A}^{-1}(\widetilde{\pi}z)$, and $f_\mathfrak{a}$'s are specific polynomials over $\mathbb{C}_\infty$ derived from the Carlitz module.

Instead of directly finding a root of $\Delta(z)$, the key idea is as follows. For $\n\in A_+$, let $\Delta_{\n} (z) = \Delta(\n z)$. First, by \cite[Corollary 3.18]{gekeler_1997}, there is a maximal $k$-th root $D_{\n}$ (up to a constant multiple) of $\frac{\Delta}{\Delta_\n}$ in $\mathcal{O}(\Omega)^\ast$ with 
$$k = \begin{cases}
(q-1)(q^2-1),  & \text{if $\deg(\n)$ is even;} \\
(q-1)^2,  & \text{otherwise.}
\end{cases}$$
By \cite[Corollary 3.21]{gekeler_1997} and \cite[Lemma 2.2]{ho_rational_2024}, we have the transformation law under $\Gamma_0(\n\m)$ of $D_\n(\m z)$ for any $\m\in A_+$. Second, if $\p$ and $\q$ are distinct primes in $A_+$ of odd degree, then there exists a $(q-1)(q^2-1)$th root $f$ of $\frac{\Delta_\p}{\Delta_\q}$ in $\mathcal{O}(\Omega)^\ast$ by \cite[Remark 6.4]{papikian_eisenstein_2015}. In Lemma \ref{The transformation law for D_pq}, we establish the transformation law under $\Gamma_0(\p\q)$ of $f$. From the above, we explicitly construct a $(q-1)(q^2-1)$th root (up to a constant multiple) of a Delta quotient on $X_0(\n)$ in Section \ref{Sec: Proof of  Theorem} whenever such a root exists as a modular function on $X_0(\n)$.

In addition, we propose the following conjecture.
\begin{conj} (cf. {\cite[Theorem 3.6]{YOO_2023}}) \label{Kevin's conjecture}
Fix $\n\in A_+$. Let $$g_{\underline{r}}(z) := \prod_{\substack{\mathfrak{a}|\n\\ \mathfrak{a}\in A_+}}\Delta_{\mathfrak{a}}^{r_{\mathfrak{a}}},$$ where $\underline{r} = (r_{\mathfrak{a}})$ is a family of integers $r_{\mathfrak{a}}\in\mathbb{Z}$ indexed by monic $\mathfrak{a}\mid\n$. If $g_{\underline{r}}$ has a $k$-th root as a modular function on $X_0(\n)$, then $\frac{r_{\mathfrak{a}}}{k}\in \frac{1}{(q-1)(q^2-1)}\mathbb{Z}$ for all monic $\mathfrak{a}\mid\n$. In particular, if $\gcd(r_{\mathfrak{a}}:\text{~monic~}\mathfrak{a}|\n) = 1$, then $k\mid (q-1)(q^2-1)$.
\end{conj}

Evidence for this conjecture is provided by the following theorem, which is proved in Section \ref{Sec: Applications}, relying on the results of \cite{ho_rational_2024} and \cite{papikian_eisenstein_2015} in the cases $\n = \p^r$ or $\n = \p\q$, respectively.

\begin{thm} \label{Evidence for Kevin's Conjecture}
Conjecture \ref{Kevin's conjecture} is true for the cases $\n = \p^r$ or $\n = \p\q$, where $\p$ and $\q$ are distinct primes in $A_+$ and $r\geq 1$.
\end{thm}

Fix $\n = \prod_{1\leq i\leq s}\p_i^{r_i}\in A_+$, where $\p_i$ are primes in $A_+$ and $r_i\geq 1$. By Gekeler \cite[(3.6)]{gekeler_1997}, the cusps of $X_0(\n)$ are in bijection with $\Gamma_0(\n)\backslash \mathbb{P}^1(K)$. Moreover, every cusp of $X_0(\n)$ has a representative $\begin{bmatrix} a \\ d \end{bmatrix}^{\n} := \begin{bmatrix} a \\ d \end{bmatrix}\in \Gamma_0(\n)\backslash \mathbb{P}^1(K)$, where $a,d\in A_+$, $d\mid \n$, and $\gcd(a,\n)=1$. Such a cusp is called of height $d$, independent of the choice of $a$. For a monic divisor $d$ of $\n$, let
$$[d]:=\begin{bmatrix} 1 \\ d \end{bmatrix}^{\n},$$ which is a cusp of $X_0(\n)$ of height $d$. In particular, $$[0]:=[1]:=\begin{bmatrix} 1 \\ 1 \end{bmatrix}^{\n}\text{~and~}[\infty]:=[\n]:=\begin{bmatrix} 1 \\ \n \end{bmatrix}^{\n}$$ are $K$-rational cusps of $X_0(\n)$ of height $1$ and $\n$, respectively. Denote $(P(\n)_d)$ the sum of all the cusps of $X_0(\n)$ of height $d\mid \n$. Then $(P(\n)_d)$ is $K$-rational in the sense that it is invariant under $\Gal(\overline{K}/K)$; see \cite[Proposition 6.3]{gekeler_invariants_2001}.

In Section \ref{Sec: Applications}, we prove the following.
\begin{prop}[cf. {\cite[Lemma 2.2]{ho_rational_2024}}] \label{lem: degree of the rational cuspidal divisor}
Let $\n = \prod_{i=1}^s \p_i^{r_i} \in A_+$ with $s\geq 1$. The degree of the rational cuspidal divisor $(P(\n)_d)$ of height $d = \prod_{i=1}^s \p_i^{h_i} \mid \n$ on $X_0(\n)$ is
$$
\begin{cases}
\displaystyle\frac{1}{q-1}\prod_{\substack{i=1\\ 0 < h_i < r_i}}^s (|\p_i|-1)|\p_i|^{\min\{h_i, r_i-h_i\}-1},  & \text{if there exists $0 < h_i < r_i$;}\\
1,  & \text{otherwise.}
\end{cases}
$$
\end{prop}

\begin{rem}
A cusp of $X_0(\n)$ of height $d$ is rational iff $\deg(P(\n)_d) = 1$. Thus, the above lemma recovers \cite[Lemma 3.1 (iii)]{papikian_eisenstein_2016}.
\end{rem}

Let $\CC(\n)$ be the rational cuspidal divisor class group of $X_0(\n)$ in \cite{ho_rational_2024}, which is the quotient of
\begin{align*}
&\Div_{\text{cusp}}^0(X_0(\n))(K)\\
&:=\text{the group of the degree $0$ rational cuspidal divisors on $X_0(\n)$}\\
&:= \left\{ C := \sum_{\substack{d|\n\\d\in A_+}}a_d\cdot (P(\n)_d)\middle| \deg(C) := \sum_{\substack{d|\n\\d\in A_+}}a_d\cdot \deg(P(\n)_d) =0, a_d\in \mathbb{Z}\right\}
\end{align*}
by its subgroup of divisors of modular units on $X_0(\n)$. Assuming Conjecture \ref{Kevin's conjecture} together with Theorem \ref{Main Theorem}, one obtains a complete description of all modular units on $X_0(\n)$. This yields new insights into the determination of the structure of $\CC(\n)$ for general $\n$. The group $\CC(\n)$ provides a lower bound for the rational torsion subgroup $\TT(\n)$ of the Jacobian $J_0(\n)$ of $X_0(\n)$. Once the structure of $\CC(\n)$ is known, it makes further progress toward the generalized Ogg conjecture, which asserts that $\CC(\n) = \TT(\n)$; cf. \cite{ArmanaHoPapikian2025}, \cite{Pal2005}, \cite{papikian_eisenstein_2015}, \cite{papikian_rational_2017}, and \cite{ho_torsion_2024}.

\begin{thm} [cf. {\cite[Proposition 4.5]{papikian_rational_2017}} and {\cite[Lemma 3.6]{ho_rational_2024}}] \label{thm: exponent of C(n)}
Let $\n = \prod_{i=1}^s \p_i^{r_i} \in A_+$ with $s\geq 1$. Then the exponent of $\CC(\n)$ divides $$\prod_{i=1}^s (|\p_i|^2-1)|\p_i|^{r_i-1}.$$
\end{thm}

\begin{rem}
The structure of $\mathcal{C}(\n)$ has been completely determined in the prime level case $\n=\p$ by \cite{Pal2005}, and in the prime power level case $\n=\p^r$ by \cite{ho_rational_2024}. After completing this work, I was informed that Conjecture \ref{Kevin's conjecture} holds for square-free $\n$, and the structure of $\mathcal{C}(\n)$ is fully determined in this case; see [G.-T. Chen, Ph.D. thesis] and \cite{papikian_eisenstein_2015}.
\end{rem}

The paper is organized as follows. In Section \ref{Sec: The transformation law}, in order to prove Theorem \ref{Main Theorem}, we compute the transformation law under $\Gamma_0(\p\q)$ of a $(q-1)(q^2-1)$th root of $\frac{\Delta_\p}{\Delta_\q}$ in $\mathcal{O}(\Omega)^\ast$, where $\p$ and $\q$ are two distinct primes in $A_+$ of odd degree. Then we prove Theorem \ref{Main Theorem} in Section \ref{Sec: Proof of Theorem}. In Section \ref{The key map g}, we construct a bridge map
$g: \Div_{\text{cusp}}^0(X_0(\n))(K) \hookrightarrow \mathcal{O}(X_0(\n))^\ast \displaystyle\tens_{\mathbb{Z}} \mathbb{Q}$, where $\mathcal{O}(X_0(\n))^\ast$ is the group of modular units on $X_0(\n)$, extending a result from the prime power case in \cite[Section 2.1]{ho_rational_2024} to the general $\n\in A_+$ case. In Section \ref{Sec: Applications}, as applications of Theorem \ref{Main Theorem} and the map $g$, we prove Theorem \ref{Evidence for Kevin's Conjecture}, Lemma \ref{lem: degree of the rational cuspidal divisor}, and Theorem \ref{thm: exponent of C(n)}.

\section{The transformation law for $D_{\p\q}(\gamma' z)$} \label{Sec: The transformation law}
Fix $\n \in A_+$ with $\delta:=\deg(\n)$. Recall that $\Delta(z)$ is the Drinfeld discriminant function, and denote $\Delta_{\n} (z) = \Delta(\n z)$. Let $D_{\n}$ be the function defined in \cite[p. 200]{gekeler_1997}. By \cite[Corollary 3.18]{gekeler_1997}, it is a maximal $k$-th root (up to constant multiple) of $\frac{\Delta}{\Delta_\n}$ in $\mathcal{O}(\Omega)^\ast$, where $$k = \begin{cases}
(q-1)(q^2-1),  & \text{if $\delta$ is even;} \\
(q-1)^2,  & \text{otherwise.}
\end{cases}$$
Let $\chi_{\n}: \Gamma_0(\n)\rightarrow \mathbb{F}_q^\ast$ be the character defined in \cite[Theorem 3.20]{gekeler_1997}.

\begin{lem} [{\cite[Corollary 3.21]{gekeler_1997}}] \label{original trans law of D}
The function $D_{\n}$ transforms under $\Gamma_0(\n)$ according to the character $$\omega_{\n}:=
\begin{cases}
\chi_{\n}\cdot \det^{\delta/2},  & \text{if $\delta$ is even;} \\
\chi_{\n}^2\cdot \det^{\delta},  & \text{otherwise.}
\end{cases}$$
\end{lem}

From the above, one obtains the following.
\begin{lem} [{\cite[Lemma 2.6]{ho_rational_2024}}] \label{extended trans law of D}
Fix $\m\in A_+$ and $\gamma\in \Gamma_0(\n\m)$. Then $$D_\n(\m \gamma z) = \omega_{\n}(\gamma)D_\n(\m z).$$
\end{lem}

Let $\p$ and $\q$ be distinct primes in $A_+$ of odd degree. By \cite[Remark 6.4]{papikian_eisenstein_2015}, fix $a, b\in A$ with $a\p - b\q = 1$, then we have
$$\frac{\Delta_\p}{\Delta_\q} = \text{const.~} f^{(q-1)(q^2-1)},$$ where $f(z) := D_{\p\q}(\gamma' z)\in \mathcal{O}(\Omega)^\ast$ and 
$$\gamma' := \begin{pmatrix}
\p & 1\\ b\p\q & a\p
\end{pmatrix}\in \GL_2(K).$$ In order to prove Theorem \ref{Main Theorem}, we need the following Lemma.

\begin{lem} \label{The transformation law for D_pq}
Let $f(z) := D_{\p\q}(\gamma' z)$ be as above. For $\gamma \in \Gamma_0(\p\q)$, we have
$$f(\gamma z) = \frac{\chi_\q(\gamma)}{\chi_{\p}(\gamma)}\cdot(\det(\gamma))^{\frac{\deg(\q)-\deg(\p)}{2}}\cdot f(z).$$
\end{lem}

\begin{proof}
Let $\gamma = \begin{pmatrix}
s & t \\ u\p\q & v
\end{pmatrix}\in \Gamma_0(\p\q)$. Observe that $\gamma' \gamma = \gamma'' \gamma'$, where $$\gamma'' := \begin{pmatrix}
a\p(u\q + s)-b\q(t\p+v) & t \p - u\q + v - s\\
\p\q (a^2u\p -b^2t\q + ab(s-v)) & av\p -bs\q - \p\q(au-bt)
\end{pmatrix}\in \Gamma_0(\p\q).$$
Let $r := av\p -bs\q - \p\q(au-bt)$. Then
\begin{enumerate}
    \item $r \equiv -bs\q \equiv (a\p-b\q)s = s \mod \p$. So, $\chi_{\p}(\gamma'') = \text{Nr}_\p(s)^{-1}$, where $\text{Nr}_\p: (A/\p)^\ast \rightarrow \mathbb{F}_q^\ast$ is the norm map.
    \item $r \equiv av\p \equiv (a\p-b\q)v = v \mod \q$. So, $\chi_{\q}(\gamma'') = \chi_\q(\gamma)$.
\end{enumerate}
Moreover, we have $\det(\gamma'') = \det(\gamma)$. Thus,
\begin{align*}
& D_{\p\q}(\gamma' \gamma z) = D_{\p\q}(\gamma'' \gamma' z)\\
& = \chi_{\p\q}(\gamma'')\cdot(\det(\gamma''))^{\frac{\deg(\p\q)}{2}}\cdot D_{\p\q}(\gamma' z)\\
& = \text{Nr}_\p(s)^{-1}\cdot\chi_\q(\gamma)\cdot(\det(\gamma))^{\frac{\deg(\p\q)}{2}}\cdot D_{\p\q}(\gamma' z)\\
& = \frac{\chi_\q(\gamma)}{\chi_{\p}(\gamma)}\cdot(\det(\gamma))^{\frac{\deg(\q)-\deg(\p)}{2}}\cdot D_{\p\q}(\gamma' z)
\end{align*}
since we have
$$\text{Nr}_{\p}(s) = \text{Nr}_{\p}(v)^{-1}\cdot\text{Nr}_{\p}(sv) = \chi_{\p}(\gamma)\cdot\det(\gamma)^{\deg(\p)}.$$
\end{proof}

\section{Proof of Theorem \ref{Main Theorem}} \label{Sec: Proof of  Theorem}
Fix $\n \in A_+$ with distinct prime divisors $\p_1, \cdots, \p_s$ of odd (even, resp.) degree for $i = 1,\cdots, t$ ($i = t+1,\cdots, s$, resp.). We assume that $\deg(\p_1)\leq \deg(\p_i)$ for all $2\leq i \leq t$. Let $$g_{\underline{r}}(z) := \prod_{\substack{\mathfrak{a}|\n\\\mathfrak{a}\in A_+}}\Delta_{\mathfrak{a}}^{r_{\mathfrak{a}}}(z),$$ where $\underline{r} = (r_{\mathfrak{a}})$ is a family of integers $r_{\mathfrak{a}}\in\mathbb{Z}$ indexed by monic $\mathfrak{a}\mid\n$.
\begin{defn}
A function $f: \Omega\rightarrow \mathbb{C}_\infty$ is called a Drinfeld modular form of weight $k\geq 0$ and type $\ell\in \mathbb{Z}/(q-1)\mathbb{Z}$ for $\Gamma_0(\n)$ if 
\begin{enumerate}
    \item $f(\gamma z) = (\det \gamma)^{-\ell}(cz+d)^k f(z)$ for $\gamma = \begin{pmatrix}
    a & b\\ c & d
    \end{pmatrix} \in \Gamma_0(\n)$ and $z\in \Omega$;
    \item $f$ is holomorphic (in the rigid analytic sense);
    \item $f$ is holomorphic at the cusps.
\end{enumerate}
\end{defn}
Recall that $\Delta_\mathfrak{a}(z)$ are Drinfeld modular forms on $\Omega$ of the same weight $q^2-1$ and type $0$ for $\Gamma_0(\n)$ for all monic $\mathfrak{a}\mid\n$ (cf. \cite[(1.2)]{gekeler_1997}). It follows that if $g_{\underline{r}}$ is $\Gamma_0(\n)$-invariant, then $\displaystyle\sum_{\substack{\mathfrak{a}|\n\\\mathfrak{a}\in A_+}} r_{\mathfrak{a}} = 0$.

Assume that $\displaystyle\sum_{\substack{\mathfrak{a}|\n\\\mathfrak{a}\in A_+}} r_{\mathfrak{a}} = 0$, and let $m:=\displaystyle\sum_{\substack{\text{monic~}\mathfrak{a}|\n\\\deg(\mathfrak{a})\text{~is~odd}}}r_{\mathfrak{a}}$. Then we have
\begin{align}
&g_{\underline{r}} = \left[\prod_{i=1}^t\prod_{\substack{\text{monic~}\mathfrak{a}|\n\\\deg(\mathfrak{a})\text{~is~odd}\\\p_i\mid \mathfrak{a}\\\p_j\nmid \mathfrak{a}~\forall j<i}}\left(\frac{\Delta}{\Delta_{\mathfrak{a}}}\right)^{-r_{\mathfrak{a}}}\right]\cdot\prod_{\substack{\text{monic~}\mathfrak{a}|\n\\\deg(\mathfrak{a})\text{~is~even}}}\left(\frac{\Delta}{\Delta_{\mathfrak{a}}}\right)^{-r_{\mathfrak{a}}}\\
&= \left(\frac{\Delta}{\Delta_{\p_1}}\right)^{-m}\cdot \left[\prod_{i=1}^t\prod_{\substack{\text{monic~}\mathfrak{a}|\n\\\deg(\mathfrak{a})\text{~is~odd}\\\p_i\mid \mathfrak{a}\\\p_j\nmid \mathfrak{a}~\forall j<i}}\left(\frac{\Delta_{\p_1}}{\Delta_{\p_i}}\frac{\Delta_{\p_i}}{\Delta_{\mathfrak{a}}}\right)^{-r_{\mathfrak{a}}}\right]\cdot\prod_{\substack{\text{monic~}\mathfrak{a}|\n\\\deg(\mathfrak{a})\text{~is~even}}}\left(\frac{\Delta}{\Delta_{\mathfrak{a}}}\right)^{-r_{\mathfrak{a}}}. \label{eq of g_r}
\end{align}

\begin{lem}
Assume that $\displaystyle\sum_{\substack{\mathfrak{a}|\n\\\mathfrak{a}\in A_+}} r_{\mathfrak{a}} = 0$. Then $g_{\underline{r}}$ has a $(q-1)(q^2-1)$th root $f$ (up to constant multiple) in $\mathcal{O}(\Omega)^\ast$ if and only if $q+1\mid m$.
\end{lem}

\begin{proof}
\begin{enumerate}
    \item If $t = 0$, then $m = 0$. The proof follows by taking $$f(z) := \prod_{\substack{\mathfrak{a}|\n\\\mathfrak{a}\in A_+}}D_{\mathfrak{a}}(z)^{-r_{\mathfrak{a}}}.$$
    \item If $t\geq 1$, the lemma is true since $\frac{\Delta}{\Delta_{\p_1}}$ has a maximal $(q-1)^2$-th root in $\mathcal{O}(\Omega)^\ast$ and all the remaining fractions in Equation (\ref{eq of g_r}) have a $(q-1)(q^2-1)$-th root in $\mathcal{O}(\Omega)^\ast$. When $q+1\mid m$, we take
    \begin{align*}
    f(z) := &D_{\p_1}(z)^{-\frac{m}{q+1}}\cdot\prod_{i=2}^t\prod_{\substack{\text{monic~}\mathfrak{a}|\n\\\deg(\mathfrak{a})\text{~is~odd}\\\p_i\mid \mathfrak{a}\\\p_j\nmid \mathfrak{a}~\forall j<i}}\left(D_{\p_1\p_i}(z_i)\right)^{-r_{\mathfrak{a}}}\\
    &\times\left[\prod_{i=1}^t\prod_{\substack{\text{monic~}\mathfrak{a}|\n\\\deg(\mathfrak{a})\text{~is~odd}\\\p_i\mid \mathfrak{a}\\\p_j\nmid \mathfrak{a}~\forall j<i}}\left(D_{\mathfrak{a}/{\p_i}}(\p_i z)\right)^{-r_{\mathfrak{a}}}\right]\cdot\prod_{\substack{\text{monic~}\mathfrak{a}|\n\\\deg(\mathfrak{a})\text{~is~even}}}D_{\mathfrak{a}}(z)^{-r_{\mathfrak{a}}},
    \end{align*}
    where $z_i := \frac{\p_1 z +1}{b_i \p_1 \p_i z + a_i\p_1}$ and $a_i\p_1 - b_i\p_i = 1$ for some fixed $a_i, b_i\in A$ with $i = 2,\cdots, t$.
\end{enumerate}
\end{proof}

In the following, we compute the transformation law under $\Gamma_0(\n)$ for the function $f$ obtained in the above lemma. Suppose that $t = 0$ and $m = 0$, then by Lemma \ref{original trans law of D}, for $\gamma\in \Gamma_0(\n)$ we have $$f(\gamma z) = f(z)\cdot \prod_{\substack{\mathfrak{a}|\n\\ \mathfrak{a}\in A_+}}(\chi_{\mathfrak{a}}(\gamma)\cdot \det(\gamma)^{\frac{\deg(\mathfrak{a})}{2}})^{-r_{\mathfrak{a}}}.$$
On the other hand, suppose that $t \geq 1$ and $q+1\mid m$, then by Lemma \ref{extended trans law of D} and \ref{The transformation law for D_pq}, for $\gamma\in \Gamma_0(\n)$ we have
\begin{align*}
f(\gamma z) &= f(z)\cdot (\chi_{\p_1}^2(\gamma)\cdot \det(\gamma)^{\deg(\p_1)})^{-\frac{m}{q+1}}\\
    &\times\prod_{i=1}^t\prod_{\substack{\text{monic~}\mathfrak{a}|\n\\\deg(\mathfrak{a})\text{~is~odd}\\\p_i\mid \mathfrak{a}\\\p_j\nmid \mathfrak{a}~\forall j<i}}\left(\frac{\chi_{\p_i}(\gamma)}{\chi_{\p_1}(\gamma)}\cdot(\det(\gamma))^{\frac{\deg(\p_i)-\deg(\p_1)}{2}}\right)^{-r_{\mathfrak{a}}}\\
    &\times\prod_{i=1}^t\prod_{\substack{\text{monic~}\mathfrak{a}|\n\\\deg(\mathfrak{a})\text{~is~odd}\\\p_i\mid \mathfrak{a}\\\p_j\nmid \mathfrak{a}~\forall j<i}}(\chi_{\mathfrak{a}/\p_i}(\gamma)\cdot \det(\gamma)^{\frac{\deg(\mathfrak{a}/\p_i)}{2}})^{-r_{\mathfrak{a}}}\\
    &\times\prod_{\substack{\text{monic~}\mathfrak{a}|\n\\\deg(\mathfrak{a})\text{~is~even}}}(\chi_{\mathfrak{a}}(\gamma)\cdot \det(\gamma)^{\frac{\deg(\mathfrak{a})}{2}})^{-r_{\mathfrak{a}}}\\
    & = f(z)\cdot \det(\gamma)^{\frac{1}{2}\left[\deg(\p_1)\cdot\frac{m}{q+1}-\frac{\sum_{\text{monic~}\mathfrak{a}|\n}r_{\mathfrak{a}}\cdot\deg(\mathfrak{a})}{q-1}\right]\cdot(q-1)}\cdot \prod_{\substack{\mathfrak{a}|\n\\ \mathfrak{a}\in A_+}}(\chi_{\mathfrak{a}}(\gamma))^{-r_{\mathfrak{a}}}
    \end{align*}
    since $(\chi_{\p_1}(\gamma))^{-2\cdot \frac{m}{q+1}+m} = (\chi_{\p_1}(\gamma))^{\frac{m}{q+1}\cdot (q-1)} = 1$.
Note that $$\{(\det(\gamma),\chi_{\p_1}(\gamma),\cdots,\chi_{\p_s}(\gamma))\mid \gamma \in \Gamma_0(\n)\} = (\mathbb{F}_q^\ast)^{s+1}.$$
Thus, in both cases, $f$ is invariant under $\Gamma_0(\n)$ iff for every $\gamma\in  \Gamma_0(\n)$, we have
\begin{enumerate}
    \item $$\det(\gamma)^{\frac{1}{2}\left[\deg(\p_1)\cdot\frac{m}{q+1}-\frac{\sum_{\text{monic~}\mathfrak{a}|\n}r_{\mathfrak{a}}\cdot\deg(\mathfrak{a})}{q-1}\right]\cdot(q-1)}=1,$$ i.e., $q-1\mid \sum_{\text{monic~}\mathfrak{a}|\n}r_{\mathfrak{a}}\cdot\deg(\mathfrak{a})$, and $\frac{m}{q+1}\equiv\frac{\sum_{\text{monic~}\mathfrak{a}|\n}r_{\mathfrak{a}}\cdot\deg(\mathfrak{a})}{q-1}\mod 2$.
    \item $\displaystyle\prod_{\substack{\mathfrak{a}|\n\\ \mathfrak{a}\in A_+}}(\chi_{\mathfrak{a}}(\gamma))^{-r_{\mathfrak{a}}} = 1$.
\end{enumerate}
From the above, we obtain the following.
\begin{lem} \label{preliminary lemma of main theorem}
$g_{\underline{r}}$ has a $(q-1)(q^2-1)$th root as a modular function on $X_0(\n)$ iff the following conditions hold:
\begin{enumerate}
    \item $\displaystyle\sum_{\substack{\mathfrak{a}|\n\\\mathfrak{a}\in A_+}} r_{\mathfrak{a}} = 0$.
    \item $m := \displaystyle\sum_{\substack{\text{monic~}\mathfrak{a}|\n\\\deg(\mathfrak{a})\text{~is~odd}}}r_{\mathfrak{a}}\equiv 0 \mod q+1$.
    \item $\displaystyle\sum_{\substack{\mathfrak{a}|\n\\\mathfrak{a}\in A_+}}r_{\mathfrak{a}}\cdot\deg(\mathfrak{a})\equiv 0 \mod q-1$.
    \item $\frac{m}{q+1}\equiv\frac{\sum_{\text{monic~}\mathfrak{a}|\n}r_{\mathfrak{a}}\cdot\deg(\mathfrak{a})}{q-1}\mod 2$.
    \item $\displaystyle\prod_{\substack{\mathfrak{a}|\n\\ \mathfrak{a}\in A_+}}(\chi_{\mathfrak{a}}(\gamma))^{-r_{\mathfrak{a}}} = 1$.
\end{enumerate}
\end{lem}

\begin{proof}[Proof of Theorem \ref{Main Theorem}]
It suffices to show that the conditions in Theorem \ref{Main Theorem} are equivalent to those in Lemma \ref{preliminary lemma of main theorem}. Assume that $\displaystyle\sum_{\substack{\mathfrak{a}|\n\\\mathfrak{a}\in A_+}} r_{\mathfrak{a}} = 0$, and let $$m:=\displaystyle\sum_{\substack{\text{monic~}\mathfrak{a}|\n\\\deg(\mathfrak{a})\text{~is~odd}}}r_{\mathfrak{a}}.$$
\begin{enumerate}
    \item Note that 
    \begin{align*}
    &\frac{1}{q^2-1}\sum_{\substack{\mathfrak{a}|\n\\ \mathfrak{a}\in A_+}} r_{\mathfrak{a}}\cdot |\mathfrak{a}| = \frac{1}{q^2-1}\sum_{\substack{\mathfrak{a}|\n\\ \mathfrak{a}\in A_+}} r_{\mathfrak{a}}\cdot (|\mathfrak{a}|-1)\\
    &= \sum_{\substack{\text{monic~}\mathfrak{a}|\n\\\deg(\mathfrak{a})\text{~is~odd}}} r_{\mathfrak{a}}\cdot \left(\frac{|\mathfrak{a}|-|\p_1|}{q^2-1}+\frac{|\p_1|-1}{q^2-1}\right) + \sum_{\substack{\text{monic~}\mathfrak{a}|\n\\\deg(\mathfrak{a})\text{~is~even}}} r_{\mathfrak{a}}\cdot \frac{|\mathfrak{a}|-1}{q^2-1}\\
    &= \frac{|\p_1|-1}{q-1}\cdot \frac{m}{q+1} + \sum_{\substack{\text{monic~}\mathfrak{a}|\n\\\deg(\mathfrak{a})\text{~is~odd}}} r_{\mathfrak{a}}\cdot |\p_1|\cdot\frac{(|\mathfrak{a}|/|\p_1|)-1}{q^2-1} + \sum_{\substack{\text{monic~}\mathfrak{a}|\n\\\deg(\mathfrak{a})\text{~is~even}}} r_{\mathfrak{a}}\cdot \frac{|\mathfrak{a}|-1}{q^2-1}.
    \end{align*}
    
    So, $\sum_{\substack{\mathfrak{a}|\n\\ \mathfrak{a}\in A_+}} r_{\mathfrak{a}}\cdot |\mathfrak{a}| \equiv 0 \mod (q-1)(q^2-1)$
    iff $q+1\mid m$ and
    \begin{align*}
    0 &\equiv \frac{|\p_1|-1}{q-1}\cdot \frac{m}{q+1} + \sum_{\substack{\text{monic~}\mathfrak{a}|\n\\\deg(\mathfrak{a})\text{~is~odd}}} r_{\mathfrak{a}}\cdot |\p_1|\cdot\frac{|\mathfrak{a}|/|\p_1|-1}{q^2-1} + \sum_{\substack{\text{monic~}\mathfrak{a}|\n\\\deg(\mathfrak{a})\text{~is~even}}} r_{\mathfrak{a}}\cdot \frac{|\mathfrak{a}|-1}{q^2-1}\\
    &\equiv \frac{1}{2}\left(\deg(\p_1)\cdot \frac{2m}{q+1} + \sum_{\substack{\text{monic~}\mathfrak{a}|\n\\\deg(\mathfrak{a})\text{~is~odd}}} r_{\mathfrak{a}}\cdot (\deg(\mathfrak{a})-\deg(\p_1)) + \sum_{\substack{\text{monic~}\mathfrak{a}|\n\\\deg(\mathfrak{a})\text{~is~even}}} r_{\mathfrak{a}}\cdot \deg(\mathfrak{a})\right)\\
    &= \frac{1}{2}\left(-\deg(\p_1)(q-1)\cdot \frac{m}{q+1} + \sum_{\text{monic~}\mathfrak{a}|\n} r_{\mathfrak{a}}\cdot \deg(\mathfrak{a})\right)\\
    & = \frac{1}{2}\left(-\deg(\p_1)\cdot \frac{m}{q+1} + \frac{\sum_{\text{monic~}\mathfrak{a}|\n} r_{\mathfrak{a}}\cdot \deg(\mathfrak{a})}{q-1}\right)\cdot(q-1) \mod q-1,
    \end{align*}
    i.e., $q-1\mid \sum_{\text{monic~}\mathfrak{a}|\n}r_{\mathfrak{a}}\cdot\deg(\mathfrak{a})$, and 
    $\frac{m}{q+1} \equiv \frac{\sum_{\text{monic~}\mathfrak{a}|\n} r_{\mathfrak{a}}\cdot \deg(\mathfrak{a})}{q-1} \mod 2.$
    \item Let $m' := m$ if $\deg(\n)$ is even; $m' := -m$ otherwise. Then
    \begin{align*}
    &\frac{1}{q^2-1}\sum_{\text{monic~}\mathfrak{a}|\n} r_{\mathfrak{a}}\cdot |\n/\mathfrak{a}| = \frac{1}{q^2-1}\sum_{\text{monic~}\mathfrak{a}|\n} r_{\mathfrak{a}}\cdot (|\n/\mathfrak{a}|-1)\\
    &= \sum_{\substack{\text{monic~}\mathfrak{a}|\n\\\deg(\n/\mathfrak{a})\text{~is~odd}}} r_{\mathfrak{a}}\cdot \left(\frac{|\n/\mathfrak{a}|-|\p_1|}{q^2-1}+\frac{|\p_1|-1}{q^2-1}\right) + \sum_{\substack{\text{monic~}\mathfrak{a}|\n\\\deg(\n/\mathfrak{a})\text{~is~even}}} r_{\mathfrak{a}}\cdot \frac{|\n/\mathfrak{a}|-1}{q^2-1}\\
    &= \frac{|\p_1|-1}{q-1}\cdot \frac{m'}{q+1} + \sum_{\substack{\text{monic~}\mathfrak{a}|\n\\\deg(\n/\mathfrak{a})\text{~is~odd}}} r_{\mathfrak{a}}\cdot |\p_1|\cdot\frac{(|\n/\mathfrak{a}|/|\p_1|)-1}{q^2-1} + \sum_{\substack{\text{monic~}\mathfrak{a}|\n\\\deg(\n/\mathfrak{a})\text{~is~even}}} r_{\mathfrak{a}}\cdot \frac{|\n/\mathfrak{a}|-1}{q^2-1}.
    \end{align*}
    
    So, $\sum_{\substack{\mathfrak{a}|\n\\\mathfrak{a}\in A_+}} r_{\mathfrak{a}}\cdot |\n/\mathfrak{a}| \equiv 0 \mod (q-1)(q^2-1)$
    iff $q+1\mid m'$ and
    \begin{align*}
    0 &\equiv \frac{|\p_1|-1}{q-1}\cdot \frac{m'}{q+1} + \sum_{\substack{\text{monic~}\mathfrak{a}|\n\\\deg(\n/\mathfrak{a})\text{~is~odd}}} r_{\mathfrak{a}}\cdot |\p_1|\cdot\frac{(|\n/\mathfrak{a}|/|\p_1|)-1}{q^2-1} + \sum_{\substack{\text{monic~}\mathfrak{a}|\n\\\deg(\n/\mathfrak{a})\text{~is~even}}} r_{\mathfrak{a}}\cdot \frac{|\n/\mathfrak{a}|-1}{q^2-1}\\
    &\equiv \frac{1}{2}\left(\deg(\p_1)\cdot \frac{2m'}{q+1} + \sum_{\substack{\text{monic~}\mathfrak{a}|\n\\\deg(\n/\mathfrak{a})\text{~is~odd}}} r_{\mathfrak{a}}\cdot (\deg(\n/\mathfrak{a})-\deg(\p_1)) + \sum_{\substack{\text{monic~}\mathfrak{a}|\n\\\deg(\n/\mathfrak{a})\text{~is~even}}} r_{\mathfrak{a}}\cdot \deg(\n/\mathfrak{a})\right)\\
    &= \frac{1}{2}\left(-\deg(\p_1)(q-1)\cdot \frac{m'}{q+1} + \sum_{\text{monic~}\mathfrak{a}|\n} r_{\mathfrak{a}}\cdot \deg(\n/\mathfrak{a})\right)\\
    & = \frac{1}{2}\left(-\deg(\p_1)\cdot \frac{m'}{q+1} - \frac{\sum_{\text{monic~}\mathfrak{a}|\n} r_{\mathfrak{a}}\cdot \deg(\mathfrak{a})}{q-1}\right)\cdot(q-1)\mod q-1,
    \end{align*}
    i.e., $q-1\mid \sum_{\text{monic~}\mathfrak{a}|\n}r_{\mathfrak{a}}\cdot\deg(\mathfrak{a})$, and 
    $\frac{m'}{q+1} \equiv \frac{\sum_{\text{monic~}\mathfrak{a}|\n} r_{\mathfrak{a}}\cdot \deg(\mathfrak{a})}{q-1} \mod 2.$
    \item $\displaystyle\prod_{\substack{\mathfrak{a}|\n\\ \mathfrak{a}\in A_+}} {\mathfrak{a}}^{r_{\mathfrak{a}}} \in (K^\ast)^{q-1}$ iff $\displaystyle\prod_{\substack{\mathfrak{a}|\n\\ \mathfrak{a}\in A_+}}(\chi_{\mathfrak{a}}(\gamma))^{-r_{\mathfrak{a}}} =  1$ for every $\gamma\in \Gamma_0(\n)$.
\end{enumerate}
\end{proof}

\begin{rem}
From the proof of Theorem \ref{Main Theorem}, assume that $\displaystyle\sum_{\substack{\mathfrak{a}|\n\\\mathfrak{a}\in A_+}} r_{\mathfrak{a}} = 0$, then the following are equivalent:
\begin{enumerate}
    \item $\displaystyle\sum_{\substack{\mathfrak{a}|\n\\ \mathfrak{a}\in A_+}} r_{\mathfrak{a}}\cdot |\mathfrak{a}| \equiv 0 \mod (q-1)(q^2-1)$.
    \item $\displaystyle\sum_{\substack{\mathfrak{a}|\n\\\mathfrak{a}\in A_+}} r_{\mathfrak{a}}\cdot |\n/\mathfrak{a}| \equiv 0 \mod (q-1)(q^2-1)$.
\end{enumerate}
\end{rem}

\section{The bridge map $g$ between rational cuspidal divisors and modular units on $X_0(\n)$} \label{The key map g}
In this section, fix $\n = \prod_{i=1}^s \p_i^{r_i}\in A_+$, where $\p_i\in A_+$ are primes, $s\geq 1$, and $r_i\geq 1$. The zero orders of $\Delta_{\mathfrak{a}}(z) := \Delta(\mathfrak{a}z)$, for $\mathfrak{a}\mid \n$, at the cusps of $X_0(\n)$ are defined in \cite[p. 47]{gekeler_drinfeld_1986}, and the divisor of $\Delta_{\mathfrak{a}}$ on $X_0(\n)$ is defined by $$\mathrm{div}(\Delta_\mathfrak{a}) := \sum_{[c]:\text{~cusp of }X_0(\n)} \ord_{[c]}(\Delta_{\mathfrak{a}})\cdot [c].$$ Extending a result from the prime power case in \cite[Section 2.1]{ho_rational_2024}, we construct a bridge map $g$, which relates a rational cuspidal divisor $C\in \Div_{\text{cusp}}^0(X_0(\n))(K)$ with a Delta quotient in the form
$$g_{\underline{r}} := \prod_{\substack{\mathfrak{a}\mid \n\\\mathfrak{a}\in A_+}}\Delta_{\mathfrak{a}}^{r_{\mathfrak{a}}},$$ where $r_{\mathfrak{a}}\in \mathbb{Z}$, such that 
\begin{equation}\label{Bridge map divisor relation}
\divisor(g_{\underline{r}}) := \sum_{\substack{\mathfrak{a}|\n\\\mathfrak{a}\in A_+}}r_{\mathfrak{a}}\cdot \mathrm{div}(\Delta_\mathfrak{a}) = (q-1)\prod_{i=1}^s (|\p_i|^2-1)|\p_i|^{r_i-1}\cdot C.
\end{equation}
Since $\deg(\text{div}(\Delta)) = \deg(\mathrm{div}(\Delta_\mathfrak{a})) > 0$ on $X_0(\n)$ for all monic $\mathfrak{a}\mid \n$, we have
$$\text{$\deg(\mathrm{div}(g_{\underline{r}})) = \sum_{\substack{\mathfrak{a}|\n\\\mathfrak{a}\in A_+}}r_{\mathfrak{a}}\cdot \deg(\mathrm{div}(\Delta_\mathfrak{a})) = 0$ on $X_0(\n)$ iff $\displaystyle\sum_{\substack{\mathfrak{a}|\n\\\mathfrak{a}\in A_+}} r_{\mathfrak{a}} = 0$}.$$

We write $a\pdiv b$ in $A$ to mean that $a\mid b$ in $A$ and that $b/a$ is coprime to $a$. Let $\p^r\pdiv \n$, where $\p\in A_+$ is a prime and $r\geq 0$. Consider the degeneracy map $\alpha_{\p}(\n): X_0(\n\p)\rightarrow X_0(\n)$ defined by
$$\alpha_{\p}(\n): [x~\text{mod}~\Gamma_0(\n\p)]\in Y_0(\n\p)\mapsto [x~\text{mod}~\Gamma_0(\n)] \in Y_0(\n)$$
and for a cusp $\begin{bmatrix}
    a\\ d\p^f
\end{bmatrix}^{\n\p}\in \Gamma_0(\n\p)\backslash\mathbb{P}^1(K)$ of $X_0(\n\p)$ with $$d\mid \frac{\n}{\p^r}\text{~and~}\gcd(a, \n\p) = 1,$$
$$\alpha_{\p}(\n): \begin{bmatrix}
    a\\ d\p^f
\end{bmatrix}^{\n\p}\mapsto\begin{cases}
\begin{bmatrix}
    a\\ d\p^f
\end{bmatrix}^{\n}, & \text{if~} f\leq r;\\
\\ 
\begin{bmatrix}
    a\\ d\p^{r+1}
\end{bmatrix}^{\n} = \begin{bmatrix}
    \p a + d\\ d\p^r
\end{bmatrix}^{\n}, & \text{if~} f = r+1.
\end{cases}$$
Note that when $f = r+1$, by \cite[Lemma 3.1 (i)]{papikian_eisenstein_2016} we have
\begin{align*}
&\begin{bmatrix}
    a\\ d\p^{r+1}
\end{bmatrix}^{\n} \sim \begin{pmatrix}
    1&0\\d\n&1
\end{pmatrix}\begin{bmatrix}
    a\\ d\p^{r+1}
\end{bmatrix}^{\n} = \begin{bmatrix}
    a\\ d\p^{r}(a\cdot\frac{\n}{\p^r} + \p)
\end{bmatrix}^{\n}\\
&\sim \begin{bmatrix}
    a(a\cdot\frac{\n}{\p^r} + \p)\\ d\p^{r}
\end{bmatrix}^{\n} \sim \begin{bmatrix}
    \p a + d\\ d\p^{r}
\end{bmatrix}^{\n}\in \Gamma_0(\n)\backslash\mathbb{P}^1(K)
\end{align*}
since
\begin{enumerate}
    \item $\begin{pmatrix}
        1&0\\d\n&1
    \end{pmatrix}\in \Gamma_0(\n)$.
    \item $a\cdot\frac{\n}{\p^r} + \p$ is coprime to $\n$.
    \item $\gcd(a(a\cdot\frac{\n}{\p^r} + \p), \n) = \gcd(\p a + d, \n) = 1$ and $d\p^r\mid \n$.
    \item $a(a\cdot\frac{\n}{\p^r} + \p)\equiv \p a + d \mod \gcd(d\p^r, \frac{\n}{d\p^r}) = \gcd(d, \frac{\n}{d\p^r})$.
\end{enumerate}
For a divisor $\prod_{i=1}^s\p_i^{h_i}$ of $\n = \prod_{i=1}^s\p_i^{r_i}$, define $$\rho_{\n}\left(\prod_{i=1}^s\p_i^{h_i}\right) = \begin{cases}
1, & \text{if there exists~} i \text{~such that~} 0<h_i<r_i;\\
q-1, & \text{otherwise}.
\end{cases}$$
For $\m\mid \n$ in $A_+$, consider the map
$$\alpha\vert^{\n}_{\m} := \alpha_{\q_1}(\m)\circ \alpha_{\q_2}(\q_1\m)\circ \cdots \circ \alpha_{\q_k}(\q_1\q_2\cdots\q_{k-1}\m)$$
from $X_0(\n)$ to $X_0(\m)$, where $\n = \m\cdot\prod_{i=1}^k \q_i$ and $\q_i$'s are primes in $A_+$.
\begin{lem}
Let $\n = \prod_{1\leq i\leq s}\p_i^{r_i}$ with a divisor $\m = \prod_{1\leq i\leq s}\p_i^{r_i'}$ in $A_+$. The ramification index of the cusp $\begin{bmatrix}
    a\\ d
\end{bmatrix}^{\n}$ of height $d\mid \n$ over the cusp $(\alpha\vert^{\n}_{\m})\begin{bmatrix}
    a\\ d
\end{bmatrix}^{\n} = \begin{bmatrix}
    a'\\ d'
\end{bmatrix}^{\m}$ of height $d'\mid \m$ is 
$$\frac{\rho_{\m}(d')}{\rho_{\n}(d)}\prod_{1\leq i\leq s}|\p_i|^{k_i},$$
where $$k_i = \begin{cases}
    \max\{r_i, 2 h_i\} - \max\{r_i', 2 h_i\}, &\text{if~} h_i < r_i';\\
    r_i - \min\{2h_i, r_i\}, &\text{otherwise}.
\end{cases}$$
Here, $d := \prod_{1\leq i\leq s}\p_i^{h_i}\mid \n$ and $d' := \gcd(d, \m) = \prod_{1\leq i\leq s}\p_i^{h_i'}$.
\end{lem}

\begin{proof}
By \cite{gekeler_1997}, the order of $\Delta$ at $\begin{bmatrix}
    a\\ d
\end{bmatrix}^{\n}$ is
\begin{equation} \label{first order}
\frac{q-1}{\rho_\n(d)}\prod_{1\leq i\leq s}|\p_i|^{r_i - \min \{2h_i, r_i\}},
\end{equation}
and the order of $\Delta$ at $\begin{bmatrix}
    a'\\ d'
\end{bmatrix}^{\m}$ is
\begin{equation} \label{second order}
\frac{q-1}{\rho_\m(d')}\prod_{1\leq i\leq s}|\p_i|^{r_i' - \min \{2h_i', r_i'\}}.
\end{equation}
Note that $h_i' = \min \{h_i, r_i'\}$. Then the proof is completed by dividing Equation (\ref{first order}) with Equation (\ref{second order}).
\end{proof}

\begin{lem}
Let $\n = \prod_{1\leq i\leq s}\p_i^{r_i}$ with a divisor $\m = \prod_{1\leq i\leq s}\p_i^{r_i'}$ in $A_+$. The order of $\Delta_{\m}$ at the cusp $\begin{bmatrix}
    a\\ d
\end{bmatrix}^{\n}$ of $X_0(\n)$ is 
$$\frac{q-1}{\rho_{\n}(d)}\prod_{1\leq i\leq s}|\p_i|^{\max\{h_i, r_i-h_i\}-|r_i'-h_i|},$$
where $d := \prod_{1\leq i\leq s}\p_i^{h_i}\mid \n$ and $\gcd(a, \n) = 1$.
\end{lem}

\begin{proof}
By \cite{gekeler_1997}, the order of $\Delta_{\m}$ at the cusp $(\alpha\vert^{\n}_{\m})\begin{bmatrix}
    a\\ d
\end{bmatrix}^{\n} = \begin{bmatrix}
    a'\\ d'
\end{bmatrix}^{\m}$ of $X_0(\m)$ of height $d' := \gcd(d, \m) = \prod_{1\leq i\leq s}\p_i^{h_i'}\mid \m$ is $$\frac{q-1}{\rho_{\m}(d')}\prod_{1\leq i\leq s}|\p_i|^{r_i'-\min\{2(r_i'-h_i'), r_i'\}}.$$ The order of $\Delta_{\m}$ at $\begin{bmatrix}
    a\\ d
\end{bmatrix}^{\n}$ is equal to the product of its order at $\begin{bmatrix}
    a'\\ d'
\end{bmatrix}^{\m}$ and the ramification index of $\begin{bmatrix}
    a\\ d
\end{bmatrix}^{\n}$ over $\begin{bmatrix}
    a'\\ d'
\end{bmatrix}^{\m}$.
\end{proof}

Let $a_{\n}(d, \mathfrak{a})$ be the order of $\Delta_{\mathfrak{a}}$ at a cusp of height $d$ in $X_0(\n)$. We define a matrix $\Lambda(\n)^\intercal := (a_{\n}(d, \mathfrak{a}))_{\substack{\mathfrak{a}, d\mid \n\\\text{monic}}}$, indexed by the divisors of $\n$. Then we have the column vector
$$[\text{div}(\Delta_{\mathfrak{a}})]_{\substack{\mathfrak{a}\mid \n\\\mathfrak{a}\in A_+}} = \Lambda(\n)^\intercal\times [(P(\n)_d)]_{\substack{d\mid \n\\d\in A_+}}.$$
Consider a diagonal matrix $$\widetilde{\rho}(\n) := \begin{bmatrix}
    \frac{q-1}{\rho_{\n}(1)}\\ & \ddots\\ && \frac{q-1}{\rho_{\n}(d)}\\ &&& \ddots\\ &&&& \frac{q-1}{\rho_{\n}(\n)}
\end{bmatrix}_{\substack{d\mid \n\\d\in A_+}}\in \GL_{k}(\mathbb{Q}), $$ where $k := |\{\text{divisors $d$ of $\n$ in $A_+$}\}|$. Define $\Upsilon(\n) := \widetilde{\rho}(\n)^{-1}\times \Lambda(\n)$. For a prime power $\p^r\in A_+$, we have $$\Upsilon(\p^r)^\intercal =
\begin{bmatrix}
|\p|^r & |\p|^{r-2} & \cdots & 1 & 1\\
|\p|^{r-1} & |\p|^{r-1} & \ddots & \vdots & \vdots\\
|\p|^{r-2} & |\p|^{r-2} & \ddots & |\p|^{r-2} & |\p|^{r-2}\\
\vdots & \vdots & \ddots & |\p|^{r-1} & |\p|^{r-1}\\
1 & 1 & \cdots & |\p|^{r-2} & |\p|^r
\end{bmatrix}_{0\leq i, j\leq r}
$$
is a matrix with the $(i,j)$-entries defined by
$$|\p|^{\max\{j,r-j\}-|i-j|}.$$ By \cite{ho_rational_2024}, the transpose $\Upsilon(\p^r)$ of $\Upsilon(\p^r)^\intercal$ is invertible over $\mathbb{Q}$ with
\begin{align*}
&\Upsilon(\p^r)^{-1} = \frac{1}{(|\p|^{2}-1)|\p|^{r-1}}\times \\
&\begin{bmatrix}
|\p| & -|\p| &&&& \\
-1 & |\p|^2+1 &&&&& \\
& -|\p| & \ddots & -|\p|^{m(j)}&&& \\
&&& (|\p|^2+1)|\p|^{m(j)-1} &&& \\
&&& -|\p|^{m(j)} & \ddots & -|\p|& \\
&&&&& |\p|^2+1 & -1 \\
&&&&& -|\p| & |\p|
\end{bmatrix}_{0\leq i, j\leq r},
\end{align*}
where $m(j):=\min\{j,r-j\}$, and the $(i,j)$-entry of $\Lambda(\p^r)^{-1}$ is
$$
\frac{1}{(|\p|^{2}-1)|\p|^{r-1}}\times \begin{cases}
(|\p|^2+1)|\p|^{m(j)-1},  & \text{if $1\leq i=j\leq r-1$.} \\
-|\p|^{m(j)},  & \text{if $|i-j|=1$ and $j\neq 0, r$.} \\
|\p|,  & \text{if $(i,j)=(0,0)$ or $(r,r)$.} \\
-1,  & \text{if $(i,j)=(1,0)$ or $(r-1,r)$.} \\
0,  & \text{otherwise.}
\end{cases}
$$
For general $\n = \prod_{i=1}^s \p_i^{r_i}\in A_+$, we have
\begin{enumerate}
    \item $\Lambda(\n)^\intercal = \Upsilon(\n)^\intercal\times \widetilde{\rho}(\n)$, where $$\Upsilon(\n)^{\intercal} = \bigotimes_{i=1}^s \Upsilon(\p_i^{r_i})^{\intercal}.$$
    \item $\Lambda(\n) = \widetilde{\rho}(\n)\times \Upsilon(\n)$, where $$\Upsilon(\n) = \bigotimes_{i=1}^s \Upsilon(\p_i^{r_i}).$$
    \item $\Lambda(\n)$ is invertible over $\mathbb{Q}$, where $\Lambda(\n)^{-1} = \Upsilon(\n)^{-1}\times \widetilde{\rho}(\n)^{-1}$ and $$\Upsilon(\n)^{-1} = \bigotimes_{i=1}^s \Upsilon(\p_i^{r_i})^{-1}.$$
\end{enumerate}

\begin{defn} \label{Defn of the map g}
For general $\n = \prod_{i=1}^s \p_i^{r_i}\in A_+$, we define a bridge map
\begin{equation*}
\begin{tikzcd}
g: \Div_{\text{cusp}}^0(X_0(\n))(K) \arrow[r] &\mathcal{O}(X_0(\n))^\ast \displaystyle\tens_{\mathbb{Z}} \mathbb{Q}\\
C= \displaystyle\sum_{\substack{d\mid \n\\ d\in A_+}} a_d \cdot (P(\n)_d) \arrow[r,mapsto] & \displaystyle\prod_{\substack{\mathfrak{a}\mid \n\\ \mathfrak{a}\in A_+}}\Delta_\mathfrak{a}^{r_\mathfrak{a}}\otimes \frac{1}{(q-1)\prod_{i=1}^s (|\p_i|^2-1)|\p_i|^{r_i-1}},
\end{tikzcd}
\end{equation*}
where $$[r_{\mathfrak{a}}]_{\substack{\mathfrak{a}|\n\\ \mathfrak{a}\in A_+}} := (q-1)\prod_{i=1}^s (|\p_i|^2-1)|\p_i|^{r_i-1}\times\Lambda(\n)^{-1}\times [a_d]_{\substack{d|\n\\ d\in A_+}}$$ with $[r_{\mathfrak{a}}]_{\substack{\mathfrak{a}|\n\\ \mathfrak{a}\in A_+}}$ and $[a_d]_{\substack{d|\n\\ d\in A_+}}$ written as column vectors over $\mathbb{Z}$. From our construction, $g$ is injective and satisfies Equation \eqref{Bridge map divisor relation}.
\end{defn}

\begin{Ex}
Let $\n = \p\q$ with distinct primes $\p, \q\in A_+$. We have
$$
\Lambda(\n)^\intercal = \Lambda(\n) = \begin{bmatrix}
|\p| & 1\\
1 & |\p|
\end{bmatrix} \otimes \begin{bmatrix}
|\q| & 1\\
1 & |\q|
\end{bmatrix} = \begin{bmatrix}
|\p\q| & |\p| & |\q| & 1\\
|\p| & |\p\q| & 1 & |\q|\\
|\q| & 1 & |\p\q| & |\p|\\
1 & |\q| & |\p| & |\p\q|
\end{bmatrix}
$$
and 
\begin{align*}
\Lambda(\n)^{-1}
& = \frac{1}{(|\p|^2-1)(|\q|^2-1)}\times \left(\begin{bmatrix}
|\p| & -1\\
-1 & |\p|
\end{bmatrix} \otimes \begin{bmatrix}
|\q| & -1\\
-1 & |\q|
\end{bmatrix}\right)\\
&= \frac{1}{(|\p|^2-1)(|\q|^2-1)}\times\begin{bmatrix}
|\p\q| & -|\p| & -|\q| & 1\\
-|\p| & |\p\q| & 1 & -|\q|\\
-|\q| & 1 & |\p\q| & -|\p|\\
1 & -|\q| & -|\p| & |\p\q|
\end{bmatrix}.
\end{align*}
Note that the entries of the matrices $\Lambda(\n)^\intercal$ and $\Lambda(\n)^{-1}$ are indexed by the ordered set $\{1, \q, \p, \p\q\}^2$. It follows that
\begin{enumerate}
    \item $$g([0]-[\infty]) = (\Delta^{|\p\q|-1}\Delta_{\q}^{|\q|-|\p|}\Delta_{\p}^{|\p|-|\q|}\Delta_{\p\q}^{1-|\p\q|})\otimes \frac{1}{(|\p|^2-1)(|\q|^2-1)}.$$
    \item \begin{align*}
        g([\p]-[\infty]) &= (\Delta^{-|\q|-1}\Delta_{\q}^{|\q|+1}\Delta_{\p}^{|\p\q|+|\p|}\Delta_{\p\q}^{-|\p\q|-|\p|})\otimes \frac{1}{(|\p|^2-1)(|\q|^2-1)}\\
        &=(\Delta^{-1}\Delta_{\q}\Delta_{\p}^{|\p|}\Delta_{\p\q}^{-|\p|})\otimes \frac{1}{(|\p|^2-1)(|\q|-1)}.
    \end{align*}
    \item \begin{align*}
        g([\q]-[\infty]) &= (\Delta^{-|\p|-1}\Delta_{\q}^{|\p\q|+|\q|}\Delta_{\p}^{|\p|+1}\Delta_{\p\q}^{-|\p\q|-|\q|})\otimes \frac{1}{(|\p|^2-1)(|\q|^2-1)}\\
        &=(\Delta^{-1}\Delta_{\q}^{|\q|}\Delta_{\p}\Delta_{\p\q}^{-|\q|})\otimes \frac{1}{(|\p|-1)(|\q|^2-1)}.
    \end{align*}
\end{enumerate}
\end{Ex}

\begin{Ex}
Let $\n = \p\q^2$ with distinct primes $\p, \q\in A_+$. Then
\begin{align*}
\Lambda(\n)^\intercal
&= \left(\begin{bmatrix}
|\p| & 1\\
1 & |\p|
\end{bmatrix} \otimes \begin{bmatrix}
|\q|^2 & 1 & 1\\
|\q| & |\q| & |\q|\\
1 & 1 & |\q|^2
\end{bmatrix}\right)\times\begin{bmatrix}
1\\ & q-1\\ && 1\\ &&& 1\\ &&&& q-1\\ &&&&& 1
\end{bmatrix}\\
&= \begin{bmatrix}
|\p\q^2| & |\p| & |\p| & |\q^2| & 1 & 1\\
|\p\q| & |\p\q| & |\p\q| & |\q| & |\q| & |\q|\\
|\p| & |\p| & |\p\q^2| & 1 & 1 & |\q^2|\\
|\q^2| & 1 & 1 & |\p\q^2| & |\p| & |\p|\\
|\q| & |\q| & |\q| & |\p\q| & |\p\q| & |\p\q|\\
1 & 1 & |\q^2| & |\p| & |\p| & |\p\q^2|
\end{bmatrix}\begin{bmatrix}
1\\ & q-1\\ && 1\\ &&& 1\\ &&&& q-1\\ &&&&& 1
\end{bmatrix},
\end{align*}
\begin{align*}
\Lambda(\n)
&= \begin{bmatrix}
1\\ & q-1\\ && 1\\ &&& 1\\ &&&& q-1\\ &&&&& 1
\end{bmatrix}\begin{bmatrix}
|\p\q^2| & |\p\q| & |\p| & |\q^2| & |\q| & 1\\
|\p| & |\p\q| & |\p| & 1 & |\q| & 1\\
|\p| & |\p\q| & |\p\q^2| & 1 & |\q| & |\q^2|\\
|\q^2| & |\q| & 1 & |\p\q^2| & |\p\q| & |\p|\\
1 & |\q| & 1 & |\p| & |\p\q| & |\p|\\
1 & |\q| & |\q^2| & |\p| & |\p\q| & |\p\q^2|
\end{bmatrix},
\end{align*}
and 
\begin{align*}
&\Lambda(\n)^{-1} = \left(\frac{1}{|\p|^2-1}\begin{bmatrix}
|\p| & -1\\
-1 & |\p|
\end{bmatrix} \otimes \frac{1}{(|\q|^2-1)|\q|}\begin{bmatrix}
|\q| & -|\q|\\
-1 & |\q|^2+1 & -1\\
& -|\q| & |\q|
\end{bmatrix}\right)\\
&~~~~~~~~~~~~\times \begin{bmatrix}
1\\ & q-1\\ && 1\\ &&& 1\\ &&&& q-1\\ &&&&& 1
\end{bmatrix}^{-1}\\
& = \frac{1}{(|\p|^2-1)(|\q|^2-1)|\q|}\times\\
&\begin{bmatrix}
|\p\q| & -|\p\q| &  & -|\q| & |\q| & \\
-|\p| & |\p|(|\q|^2+1) & -|\p| & 1 & -(|\q|^2+1) & 1\\
 & -|\p\q| & |\p\q| &  & |\q| & -|\q|\\
-|\q| & |\q| &  & |\p\q| & -|\p\q| & \\
1 & -(|\q|^2+1) & 1 & -|\p| & |\p|(|\q|^2+1) & -|\p|\\
 & |\q| & -|\q| &  & -|\p\q| & |\p\q|
\end{bmatrix}\begin{bmatrix}
1\\ & q-1\\ && 1\\ &&& 1\\ &&&& q-1\\ &&&&& 1
\end{bmatrix}^{-1}.
\end{align*}
We compute that
\begin{align*}
&g([0]-[\infty]) = \prod_{\substack{\mathfrak{a}\mid\n\\ \mathfrak{a}\in A_+}}\Delta_{\mathfrak{a}}^{r_{\mathfrak{a}}}\otimes \frac{1}{(|\p|^2-1)(|\q|^2-1)|\q|}\\
&= (\Delta^{|\p\q|}\Delta_{\q}^{-|\p|-1}\Delta_{\q^2}^{|\q|}\Delta_{\p}^{-|\q|}\Delta_{\p\q}^{|\p|+1}\Delta_{\p\q^2}^{-|\p\q|})\otimes \frac{1}{(|\p|^2-1)(|\q|^2-1)|\q|}.
\end{align*}
Observe that
\begin{enumerate}
    \item $\sum_{\substack{\mathfrak{a}\mid\n\\\mathfrak{a}\in A_+}}r_{\mathfrak{a}}\cdot |\mathfrak{a}| = -|\q|(|\p|^2-1)(|\q|^2-1)\equiv 0\mod (q-1)(q^2-1)$.
    \item $\sum_{\substack{\mathfrak{a}\mid\n\\\mathfrak{a}\in A_+}}r_{\mathfrak{a}}\cdot |\frac{\n}{\mathfrak{a}}| = |\q|(|\p|^2-1)(|\q|^2-1)\equiv 0\mod (q-1)(q^2-1)$.
    \item $\prod_{\substack{\mathfrak{a}\mid\n\\\mathfrak{a}\in A_+}}\mathfrak{a}^{r_{\mathfrak{a}}} = \p^{-(|\p|+1)(|\q|-1)}\q^{-2|\q|(|\p|-1)}\in (K^\ast)^{q-1}$.
\end{enumerate}
It follows that, if Conjecture \ref{Kevin's conjecture} is true, then $$\ord(\overline{[0]-[\infty]}) = |\q|\frac{(|\p|^2-1)(|\q|^2-1)}{(q-1)(q^2-1)}.$$
\end{Ex}

\section{Applications} \label{Sec: Applications}
As applications of Theorem \ref{Main Theorem} and the map $g$, we prove the following.

\begin{proof} [Proof of Theorem \ref{Evidence for Kevin's Conjecture}]
Let $\n = \p^r$ or $\n = \p\q$, where $\p$ and $\q$ are distinct primes in $A_+$ and $r\geq 1$. Consider $g: \Div_{\text{cusp}}^0(X_0(\n))(K) \rightarrow \mathcal{O}(X_0(\n))^\ast \displaystyle\tens_{\mathbb{Z}} \mathbb{Q}$.

When $\n = \p^r$, by \cite{ho_rational_2024}, roots of Delta quotients as modular units on $X_0(\n)$ are generated by $$\{g(\ord(\overline{C})\cdot C)\mid C \in \mathcal{B}\},$$ where $m:=\lfloor \frac{r-1}{2}\rfloor$ and $$\mathcal{B} := \{C_i,~C_j-|\p| C_{j+1},~D_{r-1},~D_0 \mid \substack{1\leq i\leq m\\ m+1\leq j\leq r-2}\};$$ see notation in \cite[Theorem 3.5]{ho_rational_2024}.

Similarly, when $\n = \p\q$, by \cite{papikian_eisenstein_2015}, roots of Delta quotients as modular units on $X_0(\n)$ are generated by
$$\{g(\ord(\overline{C})\cdot C)\mid C \in \mathcal{B}'\},$$ where $$\mathcal{B}' := \{c_{(-,-)},~c_{(-,+)},~c_{(+,-)},~c_{\p},~c_{\q}\};$$ see notation in \cite[Proposition 6.10 and 6.11]{papikian_eisenstein_2015}.

For each $C\in \mathcal{B}$ or $\mathcal{B}'$, one checks that $g((q-1)(q^2-1)\cdot\ord(\overline{C})\cdot C)$ is a Delta quotient, which completes the proof.
\end{proof}

\begin{proof}[Proof of Lemma \ref{lem: degree of the rational cuspidal divisor}]
Consider $C_d := (P(\n)_d)-\deg(P(\n)_d)\cdot[\infty]\in \Div_{\cusp}^0(X_0(\n))(K)$ with $$g(C_d) = \prod_{\substack{\mathfrak{a}|\n\\ \mathfrak{a}\in A_+}}\Delta_{\mathfrak{a}}^{r_{\mathfrak{a}}}\otimes \frac{1}{(q-1)\prod_{i=1}^s (|\p_i|^2-1)|\p_i|^{r_i-1}}.$$ If there exists $i$ with $0 < h_i < r_i$, then
\begin{align*}
0 = \sum_{\substack{\mathfrak{a}|\n\\ \mathfrak{a}\in A_+}}r_{\mathfrak{a}} &= \prod_{\substack{i=1\\ h_i = 0\text{~or~}r_i}}^s(|\p_i|-1)\prod_{\substack{i=1\\ 0 < h_i < r_i}}^s(|\p_i|-1)^2|\p_i|^{\min\{h_i, r_i-h_i\}-1}\\
&-\deg(P(\n)_d)\cdot (q-1)\prod_{i=1}^s(|\p_i|-1).
\end{align*}
Otherwise, we have
$$0 = \sum_{\substack{\mathfrak{a}|\n\\ \mathfrak{a}\in A_+}}r_{\mathfrak{a}} = (q-1)\prod_{i=1}^s(|\p_i|-1) - \deg(P(\n)_d)\cdot (q-1)\prod_{i=1}^s(|\p_i|-1).$$
\end{proof}

\begin{proof}[Proof of Theorem \ref{thm: exponent of C(n)}]
Consider $C\in \Div_{\text{cusp}}^0(X_0(\n))(K)$ with $$g(C) = \prod_{\substack{\mathfrak{a}\mid \n\\\mathfrak{a}\in A_+}}\Delta_\mathfrak{a}^{r_\mathfrak{a}}\otimes \frac{1}{(q-1)\prod_{i=1}^s (|\p_i|^2-1)|\p_i|^{r_i-1}}.$$ From our results in Section \ref{Sec: The transformation law} and Section \ref{Sec: Proof of  Theorem}, there is always a $(q-1)^2$th root $f$ (up to constant multiple) of $\prod_{\substack{\mathfrak{a}\mid \n\\\mathfrak{a}\in A_+}}\Delta_\mathfrak{a}^{r_\mathfrak{a}}$ in $\mathcal{O}(\Omega)^\ast$, where $f^{q-1}\in \mathcal{O}(X_0(\n))^\ast$, i.e., $f^{q-1}$ is invariant under $\Gamma_0(\n)$. We observe that
\begin{align*}
g\left(\left(\prod_{i=1}^s (|\p_i|^2-1)|\p_i|^{r_i-1}\right) \cdot C\right) &= (\text{const.~}f^{q^2-1})\otimes \frac{1}{q-1}\\ &= (\text{const.~}f^{q-1})\otimes 1.
\end{align*}
Thus, $\left(\prod_{i=1}^s (|\p_i|^2-1)|\p_i|^{r_i-1}\right) \cdot \overline{C} = 0$ in $\CC(\n)$.
\end{proof}

\bibliographystyle{acm}
\bibliography{Bibliography}

\end{document}